\documentclass[a4paper, pdftex]{amsart}
\AtBeginDocument{%
  \def\MR#1{}
}

\usepackage{tikz, tikz-cd}
\usepackage{amsmath,amssymb}
\usepackage{amsthm}
\usepackage{mathrsfs}
\usepackage[shortlabels]{enumitem}
\usepackage{mathtools}
\usepackage[all,cmtip,2cell]{xy}
\usepackage{array}
\usepackage{comment}
\usepackage{braket}
\usepackage{stmaryrd}
\usepackage{url}
\usepackage{ascmac}

\usepackage[top=30truemm,bottom=30truemm,left=35truemm,right=35truemm]{geometry}

\usepackage{fixme}
\fxsetup{status=draft}

\theoremstyle{plain}
\newtheorem{thm}{Theorem}[section]
\newtheorem{prop}[thm]{Proposition}
\newtheorem{lem}[thm]{Lemma}

\newtheorem{cor}[thm]{Corollary}
\newtheorem*{thm*}{Theorem}

\theoremstyle{definition}
\newtheorem{defi}[thm]{Definition}
\newtheorem{conj}[thm]{Conjecture}

\theoremstyle{remark}
\newtheorem{rem}[thm]{Remark}

\numberwithin{equation}{section}

\newcommand{\Z}{\mathbb{Z}}
\newcommand{\Q}{\mathbb{Q}}

\newcommand{\C}{\mathbb{C}}

\renewcommand{\P}{\mathbb{P}}

\newcommand{\V}{\mathbb{V}}

\DeclareMathOperator{\Sym}{\mathrm{Sym}}

\DeclareMathOperator{\Spec}{\mathrm{Spec}}

\DeclareMathOperator{\Proj}{\mathrm{Proj}}

\DeclareMathOperator{\OGr}{\mathrm{OGr}}

\DeclareMathOperator{\Spin}{\mathrm{Spin}}

\DeclareMathOperator{\Bl}{\mathrm{Bl}}

\DeclareMathOperator{\Hom}{Hom}
\DeclareMathOperator{\End}{End}

\DeclareMathOperator{\Ext}{Ext}

\DeclareMathOperator{\Aut}{Aut}
\DeclareMathOperator{\refl}{ref}

\DeclareMathOperator{\Tot}{Tot}

\newcommand{\mcE}{\mathcal{E}}
\newcommand{\mcF}{\mathcal{F}}
\newcommand{\mcG}{\mathcal{G}}

\newcommand{\mcO}{\mathcal{O}}

\newcommand{\mcP}{\mathcal{P}}

\newcommand{\mcS}{\mathcal{S}}
\newcommand{\mcT}{\mathcal{T}}
\newcommand{\mcU}{\mathcal{U}}
\newcommand{\mcV}{\mathcal{V}}

\newcommand{\wX}{\widetilde{X}}

\DeclareMathOperator{\coh}{coh}
\DeclareMathOperator{\Qcoh}{Qcoh}
\DeclareMathOperator{\modu}{mod}

\DeclareMathOperator{\RG}{\mathrm{R}\Gamma}

\DeclareMathOperator{\RHom}{\mathrm{RHom}}

\DeclareMathOperator{\Db}{\mathrm{D}^{\mathrm{b}}}
\DeclareMathOperator{\D}{\mathrm{D}}

\mathchardef\mhyphen="2D

\title[D-equivalence for the simple flop of type $G_2^{\dagger}$]{Derived equivalence for the simple flop of type $G_2^{\dagger}$ via tilting bundles}
\author[W.HARA]{Wahei Hara}

\address[W.Hara]{Kavli Institute for the Physics and Mathematics of the Universe (WPI), University of Tokyo, 5-1-5 Kashiwanoha, Kashiwa, 277-8583, Japan.}
\email{wahei.hara@ipmu.jp}

\date{\today}



\begin{document}
\maketitle

\begin{abstract}
The aim of this article is to prove the derived equivalence for a local model of the simple flop of type $G_2^{\dagger}$, which was found by Kanemitsu \cite{Kan22}.
This flop is the only known simple flop that comes from a non-homogeneous roof.
The proof of the derived equivalence is done by using tilting bundles, and also produces a noncommutative crepant resolution of the singularity that is derived equivalent to both sides of the flop.
\end{abstract}


\section{Introduction}

\subsection{Background}

The study of the derived categories of coherent sheaves over algebraic varieties is one of recent central topics in algebraic geometry, 
and it is also related to mathematical physics, representation theory, and many other fields of mathematics.
In the context of algebraic geometry, the relation between derived categories and birational geometry is actively studied, 
and many of those works are motivated by the following conjecture,
which suggests that derived categories and birational geometry would be parallel in some sense.
\begin{conj}[DK conjecture \cite{BO02, Kaw02}]
Assume that two smooth varieties $X_+$ and $X_-$ are K-equivalent, i.e.~there exist another smooth variety $\wX$ and two projective birational morphisms $\phi_{\pm} \colon \wX \to X_{\pm}$ such that $\phi_+^* K_{X_+} \sim \phi_-^* K_{X_-}$, where $K_{X_{\pm}}$ are the canonical divisors.
Then $X_+$ and $X_-$ are D-equivalent, i.e.~there exists an exact equivalence $\Db(\coh X_+) \simeq \Db(\coh X_-)$ of triangulated categories.
\end{conj}
A flop between two smooth varieties gives an example of a K-equivalent pair.
The DK conjecture is still widely open, even for flops.
It is known by \cite{BO95} that, the most standard example of a flop called an Atiyah flop satisfies the DK conjecture.

Generalising the construction of the Atiyah flops, 
Kanemitsu \cite{Kan22} found a large list of flops called simples flops, and labeled them using the Dynkin data.
The aim of this article is to prove the DK conjecture for an exceptional example of a simple flop, called the simple flop of type $G_2^{\dagger}$.

\subsection{The simple flop of type $G_2^{\dagger}$ and the main result} \label{sect: flop}
This section explains the geometry of the simple flop of type $G_2^{\dagger}$ and then states the main result.

Let $H_+ = \Q^5$ be the five dimensional quadric.
It admits the (dual) spinor bundle $\mcS^{\vee}$ of rank $4$.
It is known by Ottaviani \cite{Ottaviani88} that a general section $\mcO \to \mcS^{\vee}$ is nowhere vanishing, 
and hence it has a rank three quotient bundle $\mcG$ that lies in an exact sequence $0 \to \mcO \to \mcS^{\vee} \to \mcG \to 0$.
The bundle $\mcG$ is a stable bundle with Chern classes $(c_1, c_2,c_3) = (2,2,2)$.
It is also known that all stable rank three bundles over $\Q^5$ with these Chern classes are obtained in the way above.
These rank three bundles are sometimes called \textit{Ottaviani bundles} over $\Q^5$.
Ottaviani bundles form a non-trivial moduli, but  the group $\Aut(\Q^5)$ acts on this moduli space transitively \cite{Ottaviani88, Ott90}.

Let $\mcG_+$ be an Ottaviani bundle over $H_+$, and consider the total space
\[ X_+ := \Tot_{H_+}(\mcG_+^{\vee}(-1)) = \V_{H_+}(\mcG_+(1)) = \Spec_{H_+} \Sym^{\bullet}(\mcG_+(1)). \]
The above fact about the moduli space says, as an abstract variety, $X_+$ does not depend on the choice of Ottaviani bundles.
Let $H_+^0 \subset X_+$ be the zero section, 
$\phi_+ \colon \wX := \Bl_{H_+^0} X_+ \to X_+$ the blowing-up of $X_+$ along $H_+^0$,
and $E \subset \wX$ the exceptional divisor.
By construction, $E \simeq \P_{H_+}(\mcG_+(1)) := \Proj \Sym^{\bullet}(\mcG_+(1))$.
It is known that $E$ admits another $\P^2$-bundle structure $E \to H_-$, and the base $H_-$ is also isomorphic to $\Q^5$
(see Section~\ref{sect: roof}).
Under this, $E$ is identified with the projectivization $\P_{H_-}(\mcG_-(1))$ of an Ottaviani bundle over $H_- \simeq \Q^5$.
This geometry also yields another smooth blow-down $\phi_- \colon \wX \simeq \Bl_{H_-^0} X_- \to X_-$ with the same exceptional divisor $E$, where
$X_- = \Tot_{H_-}(\mcG_-^{\vee}(-1))$ and $H_-^0 \subset X_-$ is the zero-section.
Put $R := H^0(X_+, \mcO_{X_+}) \simeq H^0(X_-, \mcO_{X_-})$.
Then two natural morphisms $f_{\pm} \colon X_{\pm} \to \Spec R$ are flopping contractions,
and the birational map $X_+ \dashrightarrow X_-$ is the associated flop.
This is an example of \textit{a simple flop}, i.e.~a flop connected by one smooth blow-up and one smooth blow-down.
The following is the main theorem of this article.

\begin{thm} \label{main thm}
There exists an exact equivalence of $R$-linear triangulated categories 
\[ \Phi \colon \Db(\coh X_+) \xrightarrow{\sim} \Db(\coh X_-) \]
such that $R{f_+}_* \simeq  R{f_-}_* \circ \Phi$.
\end{thm}

\begin{rem}
Since $X_+$ and $X_-$ are isomorphic to each other as abstract varieties, 
the existence of a $\C$-linear equivalence $\Db(\coh X_+) \simeq \Db(\coh X_-)$ is obvious.
Thus the question in this case is the existence of an equivalence that satisfies the compatibility with the flop,
i.e.~the $R$-linear property and the commutativity $R{f_+}_* \simeq  R{f_-}_* \circ \Phi$.
\end{rem}

The classification of simple flops is related to the classification of \textit{roofs}, i.e.~Fano manifolds of Picard rank two and index $n+1$, admitting two different $\P^n$-bundle structures, for some $n$.
A roof appears in the geometry of a simple flop as the exceptional divisor of the blow-ups.
Roofs are partially classified in \cite{Kan22}, and most of them are rational homogeneous manifolds of Picard rank two with two projections to those of Picard rank one.
The roof $E$ (with two projections $E \to H_{\pm}$) is the only known example of a roof such that $E$ is not a rational homogeneous manifold.
The roof $E$ and the associated flop are called of type $G_2^{\dagger}$ in \cite{Kan22}.
In this terminology, the flop $X_+ \dashrightarrow X_-$ this article studies is a local model of the \textit{simple flop of type $G_2^{\dagger}$}.
The non-homogeneous property makes the simple flop of this type the most remarkable among all known simple flops.

\subsection{The method using tilting bundles} \label{sect: method}
The proof of the main theorem is done by using \textit{tilting bundles} (see Definition~\ref{def:tilting}).
Some previous works \cite{Hara17,Hara21, Seg16} also prove the derived equivalence for local models of the Atiyah flops, the Mukai flops, and the simple flops of type $C_2$ and $G_2$ by constructing tilting bundles.

Recall that, if a tilting bundle $\mcT$ exists over a variety $Z$ that is projective over an affine variety,
then it yields an exact equivalence of triangulated categories
\[ \RHom_Z(\mcT, -) \colon \Db(\coh Z) \xrightarrow{\sim} \Db(\modu A), \]
where $A := \End_Z(\mcT)$ and $\modu A$ is the category of finitely generated right modules over $A$
(cf.~\cite[Lemma 2.3]{TU10}).
Thus, in order to prove the derived equivalence $\Db(\coh X_+) \simeq \Db(\coh X_-)$, 
it is enough to show that there exist tilting bundles $\mcT_{\pm}$ over $X_{\pm}$ such that $\Lambda := \End_{X_+}(\mcT_+) \simeq \End_{X_-}(\mcT_-)$ as $R$-algebras.

Sections \ref{sect: tilting} and \ref{sect: D-eq} give a construction of such tilting bundles, which can be summarized as follows.
The tilting bundles constructed in this article are closely related to Kapranov's tilting bundle over $\Q^5$ \cite{Kap88}; however, the pull-back of Kapranov's bundle to $X_{\pm}$ does not yield a tilting bundle.
More precisely, computations using the Borel--Bott--Weil theorem reveal that the bundle has non-vanishing $\Ext^1$, which violates the tilting property (Lemma~\ref{pretilting}).
However, the only non-vanishing extension appears as the $\Ext^1$ between two indecomposable summands, and thus taking the corresponding extension produces tilting bundles (Theorem~\ref{thm tilting}).
Later, it turns out that the constructed tilting bundle admits a natural description from the geometry of rational homogeneous manifolds of Dynkin type $D_4$ (Proposition~\ref{prop exchange}).
In fact, this final step is the most crucial ingredient of this article, 
and allows us to compare the endomorphism rings $\End_{X_{\pm}}(\mcT_{\pm})$ on both sides.

\newpage
\subsection{Further remarks and discussions}
\subsubsection{The associated K3 surfaces}
Let $s_+$ be a regular section of the twisted Ottaviani bundle $\mcG_+(1)$ over $H_+ \simeq \Q^5$.
Then its zero-locus $S_+ \subset H_+$ is a K3 surface  of degree $12$.
Under the flop $X_+ \dashrightarrow X_-$, there exists a corresponding section 
$s_- \in H^0(H_-, \mcG_-(1))$ whose zero-locus $S_- \subset H_-$ is also a K3 surface of degree $12$ (cf.~\cite[Section 5]{Ueda19}).
If one chooses $s_+$ general enough, $S_+$ and $S_-$ give a non-isomorphic pair of K3 surfaces (\cite[Corollary 4.10]{KR22}). 
The derived equivalence in Theorem~\ref{main thm} also produces a derived equivalence $\Db(\coh S_+) \simeq \Db(\coh S_-)$ (cf.~\cite[Section 5]{Ueda19}), which gives a new proof for the D-equivalence of $S_{\pm}$.
More detailed study of those K3 surfaces from the points of view of Hodge theory and motive theory
was already done in \cite{KR22}.
The result in this article adds a new aspect from the flop to the study of those K3 surfaces.

\subsubsection{The generalised McKay correspondence}
The proof of the derived equivalence using tilting bundles fits into the picture of the generalised McKay correspondence.
Indeed, if tilting bundles $\mcT_{\pm}$ as in Section \ref{sect: method} exist, 
the noncommutative $R$-algebra $\Lambda = \End_{X_+}(\mcT_+) = \End_{X_-}(\mcT_-)$ is a \textit{noncommutative crepant resolution (NCCR)} of $R$ (cf.~\cite{VdB04,VdB23}).

It is widely expected that, if an affine scheme $\Spec S$ admits both (projective) crepant resolutions and an NCCR $A$, 
any (projective) crepant resolution $X$ of $S$ can be realised as the King's moduli space $\mathcal{M}_{\theta}(\mathbf{v})$ of $\theta$-stable finite length right $A$-modules of a fixed dimension vector $\mathbf{v}$ for some stability condition $\theta$.
This kind of expectation is called the Craw--Ishii type problem (cf.~\cite[Conjecture 1.7]{Wem18}).

If the expectation actually holds for our $R$ and an NCCR $A$ of $R$, 
then the universal bundles $\mcU_{\pm}$ over $X_{\pm}$ should be tilting bundles with $\End_{X_{\pm}}(\mcU_{\pm}) \simeq A$,
and hence $\mcU_{\pm}$ give a pair of tilting bundles with the desired property.

In order to establish the Craw--Ishii type problem for the NCCR $\Lambda$ of $R$ that is constructed in this article, one needs more detailed information about the algebra $\Lambda$, like an explicit description as the path algebra of a quiver with relations.
However, a question of that kind is not easy in general, and we leave it as a future work.

\subsubsection{Possible alternative proofs for the derived equivalence}
The previous works \cite{Morimura22, RX24, Ueda19} proved the derived equivalence for some simple flops (the Mukai flops and simple flops of types $A^G_4$, $C_2$, $D_5$ and $G_2$) using \textit{mutations of SODs} of $\Db(\coh \wX)$.
The work \cite{Xie24}, which was announced around the same time as this article, 
applies the strategy using mutation to the simple flop of type $G_2^{\dagger}$ (and $D_4$) to give an alternative proof of the derived equivalence.
On one hand,
the resulting derived equivalence $\Psi \colon \Db(\coh X_+) \xrightarrow{\sim} \Db(\coh X_-)$ from the mutation method is very complicated in general.
For example, it is unclear if one can modify the equivalence $\Psi$ so that $R{f_+}_* \simeq  R{f_-}_* \circ \Psi$ holds.
On the other hand, the advantage of the mutation method is that it works in any models of the flop and in the relative settings.

Other possible methods to prove the derived equivalences are the method finding an explicit Fourier--Mukai kernel and the method using the variation of GIT.
For the first method, the structure sheaf of the fiber product $X_+ \times_{\Spec R} X_-$ is a candidate of the Fourier--Mukai kernel, and it actually gives an equivalence if a simple flop is an Atiyah flop, a Mukai flop, or of type $C_2$ \cite{BO95, Hara22,Kaw02,Nam03}.
The method using the variation of GIT is also expected to work for all simple flops, 
but at the time of writing, this strategy is known to work only for the Atiyah flops and the Mukai flops \cite{CF21, HS20}.
One problem is that, for most of simple flops including of type $G_2^{\dagger}$, a nice GIT picture has not been found yet.

\section{The geometry of the roof}

The present section explains the geometry of roofs, 
and constructs a key vector bundle that will be used in the proof of the main theorem.

\subsection{The roofs of types $D_4$ and $G_2^{\dagger}$} \label{sect: roof}

Consider the orthogonal Grassmannians $\OGr(3,8)$ and $\OGr(4,8)$.
It is known that $\OGr(4,8)$ has two connected components that are denoted by $\OGr_+(4,8)$ and $\OGr_-(4,8)$,
and both are isomorphic to the six dimensional quadric $\Q^6 \subset \P^7$.
There are natural projections $p_{\pm} \colon \OGr(3,8) \to \OGr_{\pm}(4,8)$.
Let $\mcS_{\pm}$ be the restriction of rank $4$ universal subbundle to $\OGr_{\pm}(4,8)$.
Under an identification $\OGr_{\pm}(4,8) \simeq \Q^6$, $\mcS_{\pm}$ is isomorphic to a spinor bundle over $\Q^6$.
In addition, the projections $p_{\pm}$ give identifications
\[ 
\begin{tikzcd}
  \P_{\OGr_+(4,8)}(\mcS_+(1)) \arrow[d, "p_+"'] \arrow[r, equal] & \OGr(3,8) \arrow[r, equal] & \P_{\OGr_-(4,8)}(\mcS_-(1)) \arrow[d, "p_-"]  \\
 \OGr_+(4,8) && \OGr_-(4,8).
\end{tikzcd} \]
This diagram is called the roof of type $D_4$, since it consists of rational homogeneous manifolds of Dynkin type $D_4$ \cite{Kan22}.

Of course, twisting a vector bundle by a line bundle does not change the variety obtained by projectivization, 
so $\OGr(3,8)$ can also be identified with $\P_{\OGr_{\pm}(4,8)}(\mcS_{\pm})$.
However, by viewing $\OGr(3,8)$ as above, we obtain isomorphisms of line bundles
\begin{center}
$p_+^*\mcO_{\OGr_+(4,8)}(1) \simeq \mcO_{p_-}(1)$ and $p_-^*\mcO_{\OGr_-(4,8)}(1) \simeq \mcO_{p_+}(1)$,
\end{center}
where $\mcO_{p_+}(1)$ and $\mcO_{p_-}(1)$ are the tautological line bundles for the $\P^3$-bundle structures $p_+$ and $p_-$, respectively.

Choosing a general member $D_+$ of the linear system $\lvert \mcO_{p_+}(1) \rvert$ corresponds to choosing a general section $\mcO_{\OGr_+(4,8)} \to \mcS_+(1)$, and it is known by \cite[Section 3]{Ottaviani88} that this section has a locally free quotient bundle $\mcS_+(1)/\mcO$ of rank three.
On the other hand, $D_+$ also corresponds to a hyperplane section $H_- \simeq \Q^5 \subset \OGr_-(4,8) \simeq \Q^6$.
On $H_-$, an isomorphism $\mcS_-(1)|_{H_-} \simeq \mcS_-^{\vee}|_{H_-}$ exists.
Thus $D_+$ fits in a diagram
\[ 
\begin{tikzcd}
  \P_{\OGr_+(4,8)}(\mcS_+(1)/\mcO) \arrow[d, "q_+"'] \arrow[r, equal] & D_+ \arrow[r, equal] & \P_{H_-}(\mcS_-^{\vee}|_{H_-}) \arrow[d, "q_-"]  \\
 \OGr_+(4,8) && H_-.
\end{tikzcd} \]
Similarly, choosing a general member $W$ of the linear system $\lvert \mcO_{q_-}(1) \rvert$ corresponds to choosing a general section $\mcO_{H_-} \to \mcS_-^{\vee}|_{H_-}$, 
and again, this section has a rank three locally quotient $\mcG_-$, 
which is an Ottaviani bundle over $H_- \simeq \Q^5$.
On the other hand, $W$ also corresponds to a hyperplane section $H_+ \simeq \Q^5 \subset \OGr_+(4,8) \simeq \Q^6$,
and the restriction $\mcG_+ :=  (\mcS_+(1)/\mcO)|_{H_+}$ is also an Ottaviani bundle over $H_+ \simeq \Q^5$.
The obtained diagram
\[ 
\begin{tikzcd}
  \P_{H_+}(\mcG_+) \arrow[d, "\pi_+"'] \arrow[r, equal] & W \arrow[r, equal] & \P_{H_-}(\mcG_-) \arrow[d, "\pi_-"] \\
 H_+ \simeq \Q^5 && H_- \simeq \Q^5
\end{tikzcd} \]
is called \textit{the roof of type $G_2^{\dagger}$}.

The variety $W$ can also be understood as a smooth intersection $D_+ \cap D_-$ of two smooth divisors $D_{\pm} \in \lvert \mcO_{p_{\pm}}(1) \rvert$.
At the time of writing, this variety $W$ is the only known example of a roof that is not a rational homogeneous variety.

\subsection{The key bundle and the exchange diagram over $\OGr(3,8)$ and $W$} \label{key bundle}
The aim of this section is to produce a vector bundle $\mcE$ over $W$, which plays an important role during the proof of the main theorem.
Put $Z = \OGr(3,8)$ and $\mcO_Z(a,b) := p_+^*\mcO_{\OGr_+(4,8)}(a) \otimes p_-^*\mcO_{\OGr_-(4,8)}(b)$.
Let $\mcV$ be the rank three universal subbundle over $Z$.
Then there are two (dual of the) tautological sequences
\begin{center}
$0 \to \mcO_Z(1,-1) \to p_+^*\mcS_+^{\vee} \to \mcV^{\vee} \to 0$ and
$0 \to \mcO_Z(-1,1) \to p_-^*\mcS_-^{\vee} \to \mcV^{\vee} \to 0$.
\end{center}
Taking the fiber product of two bundle $p_+^*\mcS_+^{\vee}$ and $p_-^*\mcS_-^{\vee}$ over $\mcV^{\vee}$ gives a rank five vector bundle $\mcE$ on $Z$ and a commutative diagram
\[ \begin{tikzcd}
 & & 0 \arrow[d] & 0 \arrow[d] & \\
 & & \mcO_Z(-1,1) \arrow[r, equal] \arrow[d] & \mcO_Z(-1,1) \arrow[d] & \\
0 \arrow[r] & \mcO_Z(1,-1) \arrow[r] \arrow[d, equal] & \mcE \arrow[r] \arrow[d] & p_-^*\mcS_-^{\vee} \arrow[r] \arrow[d] & 0 \\
0 \arrow[r] & \mcO_Z(1,-1) \arrow[r] & p_+^*\mcS_+^{\vee} \arrow[r] \arrow[d] & \mcV^{\vee} \arrow[r] \arrow[d] & 0 \\
&&0&0&
\end{tikzcd} \]
of exact sequences.

Since this diagram consists of vector bundles only, all sequences are locally split, 
and hence it restricts to a diagram of exact sequences over $W \subset Z$.
Note that the restriction of $p_{\pm}^*\mcS_{\pm}^{\vee}$ to $W$ is the pull-back of the dual of the spinor bundle over $H_{\pm} \simeq \Q^5$ by the projection $\pi_{\pm} \colon W \to H_{\pm}$.
By abuse of notation, the spinor bundle over $H_{\pm} \simeq \Q^5$ is also denoted by $\mcS_{\pm}$, and the restriction of $\mcE$ to $W$ is also denoted by $\mcE$.

\begin{lem} \label{lem non-split}
Two exact sequences
\begin{align*}
&0 \to \mcO_W(-1,1) \to \mcE \to \pi_+^* \mcS_+^{\vee} \to 0, ~\text{and} \\
&0 \to \mcO_W(1,-1) \to \mcE \to \pi_-^* \mcS_-^{\vee} \to 0
\end{align*}
over $W$ do not split.
\end{lem}

\begin{proof}
First, observe that 
\begin{align*}
\Hom_W(\mcO_W(1,-1), \mcO_W(-1,1)) &\simeq H^0(W, \mcO_W(-2,2)) \\
&{}\simeq{} H^0(H_-, \mcO_{H_-}(2) \otimes (\pi_-)_*\mcO_{\pi_-}(-2)) \\
&{}={} 0.
\end{align*}
Thus, the natural homomorphism $\Ext_W^1(\mcV^{\vee}, \mcO_W(-1,1)) \to \Ext_W^1(\pi_+^*\mcS^{\vee}, \mcO_W(-1,1))$, 
which sends the class $[0\to O_W(-1,1) \to \pi_-^* \mcS_-^{\vee} \to \mcV^{\vee}|_W \to 0]$ 
to the class  $[0\to O_W(-1,1) \to \mcE \to \pi_+^*\mcS_+^{\vee} \to 0]$,
is injective.
Since the dual spinor bundle $\mcS_-^{\vee}$ is indecomposable, the exact sequence $0\to O_W(-1,1) \to \pi_-^* \mcS_-^{\vee} \to \mcV^{\vee}|_W \to 0$ does not split, and hence the corresponding class in $\Ext_W^1(\pi_+^*\mcS^{\vee}, \mcO_W(-1,1))$ is non-zero.
Thus the extension $0 \to \mcO_W(-1,1) \to \mcE \to \pi_+^* \mcS_+^{\vee} \to 0$ does not split.

Swapping the roles of $\mcS_-$ and $\mcS_+$ in the above argument proves that the second sequence does not split, either. 
\end{proof}

The above two exact sequences and the diagram will be used to exchange two tilting bundles over both sides of the flop in Section \ref{sect: D-eq}.
The diagram of the exact sequences over $W$ will be referred as the \textit{exchange diagram}.
The bundle $\mcE$ over $W$ will be referred as the \textit{key bundle}.

\section{Proof of the derived equivalence}

This section gives the proof of the main theorem.

\subsection{Borel--Bott--Weil computations}
Let $\mcS$ be the spinor bundle over $\Q^5$.
The five dimensional quadric $\Q^5$ is isomorphic to a rational homogeneous manifold $\Spin(7)/P(\alpha_1)$ of Dynkin type $B_3$, where $P(\alpha_1)$ is the parabolic subgroup of $\Spin(7)$ corresponding to the simple root $\alpha_1$ (cf.~\cite[Section 1.4]{Ott90}).
The weight lattice for the representations of $\Spin(7)$ is spanned by three weights $L_1, L_2, L_3$ and an additional weight $(L_1+L_2+L_3)/2$ (cf.~\cite[Section 19.3]{FH91}).
Thus a weight vector can be identified with $(k_1/2, k_2/2, k_3/2)$ such that $k_1, k_2, k_3$ are all odd integers or all even integers.

Fix an identification $\Q^5 \simeq \Spin(7)/P(\alpha_1)$.
If $\mcF$ is an irreducible homogeneous vector bundle over $\Q^5 \simeq \Spin(7)/P(\alpha_1)$, 
it corresponds to a weight $(k_1/2, k_2/2, k_3/2)$.
Conversely, a weight uniquely determines an irreducible homogeneous bundle, if it exists.
Thus it is allowed to denote as
\[ \mcF = \mcF\left(\frac{k_1}{2}, \frac{k_2}{2}, \frac{k_3}{2}\right). \]
Under this notation, 
\begin{center} 
$\mcO_{\Q^5}(a) = \mcF(a,0,0)$ and $\mcS^{\vee} = \mcF\left(\frac{1}{2}, \frac{1}{2}, \frac{1}{2}\right)$.
\end{center}

\begin{lem} \label{BBW}
The following hold.
\begin{enumerate}
\item[(1)] $H^i(\Q^5, \Sym^k \mcS^{\vee} \otimes \mcO_{\Q^5}(k+j)) = 0$ for all $i > 0$, $k \geq 0$ and $j \geq - 4$.
\item[(2)] $H^i(\Q^5, \Sym^k \mcS^{\vee} \otimes  \mcS^{\vee} \otimes \mcO_{\Q^5}(k+j)) = 0$ for all $i > 0$, $k \geq 0$ and $j \geq - 3$ except $(i,j,k) = (1,-3, 1)$.
\item[(3)] $H^1(\Q^5, \mcS^{\vee} \otimes \mcS^{\vee}(-2)) \simeq \C$.
\item[(4)] $H^i(\Q^5, \Sym^k \mcS^{\vee} \otimes \mcS^{\vee} \otimes \mcS \otimes \mcO(k+j)) = 0$ for all $i > 0$, $k \geq 0$ and $j \geq 0$.
\end{enumerate}
\end{lem}

\begin{proof}
Basically this lemma is a consequence of the Borel--Bott--Weil theorem (cf.~\cite[Section 4.3]{Weyman03}).
For simplicity, let $\left(\frac{k_1}{2}, \frac{k_2}{2}, \frac{k_3}{2}\right)$ denote the irreducible homogeneous bundle $\mcF\left(\frac{k_1}{2}, \frac{k_2}{2}, \frac{k_3}{2}\right)$.

For (1), note that $\Sym^k \mcS^{\vee} \otimes \mcO_{\Q^5}(k+j) = \left(\frac{3k+2j}{2}, \frac{k}{2}, \frac{k}{2}\right)$.
If $k \geq 0$ and $j \geq - 4$, the weight is dominant except when 
\[ (k,j) \in \{(0,-4), (0,-3), (0,-2), (0,-1),(1,-4), (1,-3), (1,-2), (2,-4), (2,-3), (3,-4) \}. \]
If $(k,j)$ is one of the above ten, then the bundle is one of 
\begin{center}
$(-4,0,0)$, $(-3,0,0)$, $(-2,0,0)$, $(-1,0,0)$, $\left(-\frac{5}{2}, \frac{1}{2}, \frac{1}{2}\right)$, $\left(-\frac{3}{2}, \frac{1}{2}, \frac{1}{2}\right)$, $\left(-\frac{1}{2}, \frac{1}{2}, \frac{1}{2}\right)$, $(-1,1,1)$, $(0,1,1)$, and $\left(\frac{1}{2}, \frac{3}{2}, \frac{3}{2}\right)$.
\end{center}
Under the dotted action of the Weyl group on the space of weights (cf.~\cite[Section 4.3]{Weyman03}), all the above weights have non-trivial stabilizers, 
and hence the corresponding vector bundles are acyclic. 
This shows (1).

For (2) and (3), the Littlewood--Richardson rule (cf.~\cite[Section 25.3]{FH91}) yields
\begin{align*}
&\Sym^k \mcS^{\vee} \otimes  \mcS^{\vee} \otimes \mcO_{\Q^5}(k+j) \\
&\simeq  \begin{cases}
\left(\frac{2j+1}{2}, \frac{1}{2}, \frac{1}{2}\right) & \text{if $k=0$} \\
\left(\frac{3k+2j+1}{2}, \frac{k+1}{2}, \frac{k+1}{2}\right) \oplus \left(\frac{3k+2j+1}{2}, \frac{k+1}{2}, \frac{k-1}{2}\right)
\oplus \left(\frac{3k+2j+1}{2}, \frac{k-1}{2}, \frac{k-1}{2}\right) &\text{if $k \geq 1$.}
\end{cases}
\end{align*}
If $k \geq 0$ and $j \geq -3$,
as for (1), one can see that all weights are dominant or give acyclic bundles except when $(k,j) = (1,-3)$.
In this exceptional case, the bundle we consider is $\mcS^{\vee} \otimes \mcS^{\vee}(-2)$.
Now it is well-known that
\[  \RG(\Q^5, \mcS^{\vee} \otimes \mcS^{\vee}(-2)) = \RHom_{\Q^5}(\mcS^{\vee}, \mcS) \simeq \C[-1]. \]
(Of course, one can also deduce this from the Borel--Bott--Weil theorem.)
This completes the proof of (2) and (3).

It remains to show (4).
First, note that 
\[ \mcS^{\vee} \otimes \mcS \simeq (0,1,1) \oplus (0,1,0) \oplus (0,0,0). \]
It holds that $(0,1,1) \simeq \Sym^2 \mcS^{\vee} \otimes \mcO(-1)$ and 
that $(0,1,0) \simeq T_{\Q^5}(-1)$, where $T_{\Q^5}$ is the tangent bundle.
The Littlewood--Richardson rule gives
\begin{align*}
&\Sym^k \mcS^{\vee} \otimes  \Sym^2 \mcS^{\vee} \otimes \mcO_{\Q^5}(k+j-1) \\
&\simeq  \begin{cases}
(j, 1, 1) & \text{if $k=0$} \\
\left(\frac{3+2j}{2}, \frac{3}{2}, \frac{3}{2}\right) \oplus \left(\frac{3+2j}{2}, \frac{3}{2}, \frac{1}{2}\right) \oplus \left(\frac{3+2j}{2}, \frac{1}{2}, \frac{1}{2}\right) & \text{if $k=1$} \\
\left(\frac{3k+2j}{2}, \frac{k+2}{2}, \frac{k+2}{2}\right) \oplus \left(\frac{3k+2j}{2}, \frac{k+2}{2}, \frac{k}{2}\right)
\oplus \left(\frac{3k+2j}{2}, \frac{k +2}{2}, \frac{k - 2}{2}\right) & \\
\oplus \left(\frac{3k+2j}{2}, \frac{k}{2}, \frac{k}{2}\right) \oplus \left(\frac{3k+2j}{2}, \frac{k}{2}, \frac{k-2}{2}\right)
\oplus \left(\frac{3k+2j}{2}, \frac{k -2}{2}, \frac{k - 2}{2}\right) & \text{if $k \geq 2$}
\end{cases} 
\\
&\Sym^k \mcS^{\vee} \otimes  T_{\Q^5} \otimes \mcO_{\Q^5}(k+j-1) \\
&\simeq  \begin{cases}
(j, 1, 0) & \text{if $k=0$} \\
\left(\frac{3+2j}{2}, \frac{3}{2}, \frac{1}{2}\right) \oplus \left(\frac{3+2j}{2}, \frac{1}{2}, \frac{1}{2}\right) & \text{if $k=1$} \\
\left(\frac{3k+2j}{2}, \frac{k+2}{2}, \frac{k}{2}\right) \oplus \left(\frac{3k+2j}{2}, \frac{k}{2}, \frac{k}{2}\right)
\oplus \left(\frac{3k+2j}{2}, \frac{k}{2}, \frac{k - 2}{2}\right) & \text{if $k \geq 2$}
\end{cases} \\
&\Sym^k \mcS^{\vee} \otimes  \mcO_{\Q^5}(k+j) \simeq \left(\frac{3k+2j}{2}, \frac{k}{2}, \frac{k}{2}\right).
\end{align*}
If $k \geq 0$ and $j \geq 0$, one can see that all above weights are dominant or give acyclic bundles.
This shows (4).
\end{proof}

\subsection{Construction of tilting bundles} \label{sect: tilting}

Let $\mcG$ be an Ottaviani bundle over $\Q^5$, and put
\[ X = \Tot_{\Q^5}(\mcG^{\vee}(-1)) = \Spec_{\Q^5} \Sym^{\bullet} \mcG(1). \]
The variety $X$ is a local Calabi--Yau eightfold with a flopping contraction $f \colon X \to \Spec H^0(X,\mcO_X)$.

The aim of the present section is to construct tilting bundles over the variety $X$.
Recall that the notion of tilting bundle is defined as follows.

\begin{defi} \label{def:tilting}
Let $Z$ be a scheme that is projective over an affine Noetherian scheme.
\begin{enumerate}
\item[(1)] A \textit{pretilting} bundle over $Z$ is a locally free sheaf $\mcE$ of finite rank such that $\Ext^i_Z(\mcE, \mcE) = 0$ for all $i > 0$.
\item[(2)] A \textit{tilting} bundle over $Z$ is a pretilting bundle $\mcT$ that generates the (unbounded) derived category $\D(\Qcoh Z)$ of quasi-coherent sheaves over $Z$, i.e.~if $F \in \D(\Qcoh Z)$ satisfies $\RHom_Z(\mcT, F) = 0$, then $F = 0$.
\end{enumerate}
\end{defi}

Let $\mcO_X(a)$ and $\mcS_X$ denote the pull-back of the line bundle $\mcO_{\Q^5}(a)$ and the spinor bundle $\mcS$ on $\Q^5$ by the projection $X \to \Q^5$.

\begin{lem} \label{BBW2}
The following hold.
\begin{enumerate}
\item[(1)] $H^i(\Q^5, \Sym^k(\mcG(1)) \otimes \mcO_{\Q^5}(j)) = 0$ for all $i>0$, $k \geq 0$ and $j \geq -4$.
\item[(2)] $H^i(\Q^5, \Sym^k(\mcG(1)) \otimes \mcS^{\vee}(j)) = 0$ for all $i>0$, $k \geq 0$ and $j \geq -3$, except when $(i,j,k) = (1,-3,1)$.
\item[(3)] $H^1(\Q^5, \mcG(1) \otimes \mcS^{\vee}(-3)) = \C$.
\item[(4)] $H^i(\Q^5, \Sym^k(\mcG(1)) \otimes \mcS^{\vee} \otimes \mcS) = 0$ for all $i > 0$ and $k \geq 0$.
\end{enumerate}
\end{lem}

\begin{proof}
Recall that an Ottaviani bundle $\mcG$ fits in an exact sequence $0 \to \mcO_{\Q^5} \to \mcS^{\vee} \to \mcG \to 0$.
Taking the symmetric power gives exact sequences
\[ 0 \to (\Sym^{k -1}\mcS^{\vee})(k) \to (\Sym^{k}\mcS^{\vee})(k) \to \Sym^k(\mcG(1)) \to 0 \]
for all $k \geq 0$.
Now the result follows from Lemma~\ref{BBW}.
\end{proof}

\begin{lem} \label{pretilting}
The following holds.
\begin{enumerate}
\item[(1)] $H^i(X, \mcO_X(j)) = 0$ for all $i > 0$ and $j \geq -4$.
\item[(2)] $H^i(X, \mcS_X^{\vee}(j)) = 0$ for all $i>0$ and $j \geq -3$ except when $(i,j) = (1,-3)$.
\item[(3)] $H^1(X, \mcS_X^{\vee}(-3)) = \C$.
\item[(4)] $\Ext_X^i(\mcS_X, \mcS_X) = 0$ for all $i>0$.
\end{enumerate}
\end{lem}

\begin{proof}
If $\mcF$ is a vector bundle over $\Q^5$ and $\mcF_X$ denotes the pull-back of $\mcF$ by $X \to \Q^5$,
the adjunction gives
\[ H^i(X, \mcF_X) = H^i\left(\Q^5, \mcF \otimes \bigoplus_{k \geq 0} \Sym^k(\mcG(1))\right). \]
Thus the result follows from Lemma~\ref{BBW2}.
\end{proof}

By the above lemma, it follows that $\Ext_X^1(\mcS_X^{\vee}, \mcO_X(-2)) \simeq H^1(X, \mcS_X^{\vee}(-3)) = \C$.
Let
\[ 0 \to \mcO_X(-2) \to \mcP \to \mcS_X^{\vee} \to 0 \]
be the unique non-trivial extension.

\begin{thm} \label{thm tilting}
The bundles
\begin{align*}
\mcT &:= \mcO_X(-2) \oplus \mcO_X(-1) \oplus \mcO_X \oplus \mcO_X(1) \oplus \mcO_X(2) \oplus \mcP, ~ \text{and} \\
\mcT^{\vee} &= \mcO_X(-2) \oplus \mcO_X(-1) \oplus \mcO_X \oplus \mcO_X(1) \oplus \mcO_X(2) \oplus \mcP^{\vee}
\end{align*}
are tilting bundles over $X$.
\end{thm}

\begin{proof}
By Kapranov \cite{Kap88}, it is known that the bundle
\[ \mcO_{\Q^5}(-2) \oplus \mcO_{\Q^5}(-1) \oplus \mcS \oplus \mcO_{\Q^5} \oplus \mcO_{\Q^5}(1) \oplus \mcO_{\Q^5}(2) \]
is a tilting bundle over $\Q^5$, and in particular it generates $\D(\Qcoh \Q^5)$.
Thus, by \cite[Lemma~3.1]{Hara21}, the bundle
\[ \mcU := \mcO_X(-2) \oplus \mcO_X(-1) \oplus \mcS_X \oplus \mcO_X \oplus \mcO_X(1) \oplus \mcO_X(2) \]
also generates $\D(\Qcoh X)$.

Consider the collection of vector bundles appearing as direct summands of $\mcU$, taken in the order given above.
Then, by Lemma~\ref{pretilting}, this collection consists of pretilting bundles.
Furthermore, there are no non-trivial higher $\Ext$'s to the right, 
and when viewed to the left, no non-trivial $\Ext^i$ with $i \ge 2$ appear. 
Therefore, it is able to apply \cite[Lemma~2.4]{Hara21} to construct a tilting bundle.

Moreover, the only nontrivial higher extension is $\Ext_X^1(\mcO_X(2), \mcS_X) = \C$,
and the corresponding exact sequence is
$0 \to \mcS_X \to \mcP^{\vee} \to \mcO_X(2) \to 0$.
Thus the proof of \cite[Lemma~2.4]{Hara21} shows that the resulting tilting bundle over $X$ is the bundle
\[ \mcT^{\vee} = \mcO_X(-2) \oplus \mcO_X(-1) \oplus \mcO_X \oplus \mcO_X(1) \oplus \mcO_X(2) \oplus \mcP^{\vee}. \]
The dual of this bundle $\mcT$
is also a pretilting bundle, and since $\mcT$ can generate the shifts of $\mcS_X^{\vee}$ and $\mcS_X$ by taking cones, 
$\mcT$ is also a generator of $\D(\Qcoh X)$, and hence is tilting.
\end{proof}

\subsection{Derived equivalence} \label{sect: D-eq}

This section completes the proof of Theorem~\ref{main thm}.
Recall from Section~\ref{sect: flop} that the diagram of the flop is 
\[ \begin{tikzcd}
 & E = W \arrow[d, hook] \arrow[dl, "\pi_+"'] \arrow[dr, "\pi_-"]&  \\
H_+^0 \arrow[d, hook] & \wX = \Tot_W(\mcO_W(-1,-1)) \arrow[dl, "\phi_+"', "\Bl_{H_+^0}"] \arrow[dr, "\phi_-", "\Bl_{H_-^0}"'] & H_-^0 \arrow[d, hook] \\
X_+ = \Tot_{H_+}(\mcG_+^{\vee}(-1)) \arrow[d] \arrow[rr, dashed, "\text{flop}"'] \arrow[dr, "f_+"'] & & X_- = \Tot_{H_-}(\mcG_-^{\vee}(-1)) \arrow[d] \arrow[dl, "f_-"] \\
 H_+ = \Q^5 & \Spec R & H_- = \Q^5.
\end{tikzcd} \]
As in Section~\ref{sect: tilting}, let $\mcS_{X_{\pm}}$ be the pull-back of the spinor bundle over $H_{\pm}$ by the projection $X_{\pm} \to H_{\pm}$,
and $\mcP_{\pm}$ the bundles that lie in the unique non-trivial extensions
\[ 0 \to \mcO_{X_{\pm}}(-2) \to \mcP_{\pm} \to \mcS_{X_{\pm}} \to 0. \]
Put
\[ \mcT_{\pm} := \mcO_{X_{\pm}}(-2) \oplus \mcO_{X_{\pm}}(-1) \oplus \mcO_{X_{\pm}} \oplus \mcO_{X_{\pm}}(1) \oplus \mcO_{X_{\pm}}(2) \oplus \mcP_{\pm}. \]
By Theorem~\ref{thm tilting}, $\mcT_{\pm}$ and $\mcT_{\pm}^{\vee}$ are tilting bundles over $X_{\pm}$.

\begin{thm} \label{tilting exchange}
There are isomorphisms
\begin{center}
${f_+}_*\mcT_+ \simeq {f_-}_*\mcT_-$ and ${f_+}_*\mcT_+^{\vee} \simeq {f_-}_*\mcT_-^{\vee}$
\end{center}
of $R$-modules.
\end{thm}

The proof of this theorem requires the observation from Section~\ref{key bundle}.
Let $\mcE$ be the key bundle over $W$.
Let $\mcE_{\wX}$ denote the pull-back of the key bundle $\mcE$ to $\wX$ by the projection $\wX = \Tot_W(\mcO_W(-1,-1)) \to W$.
The pull-back of the exchange diagram over $W$ to $\wX$ by the projection gives two exact sequences
\begin{align*}
&0 \to \mcO_{\wX}(-1,1) \to \mcE_{\wX} \to \phi_+^*\mcS_{X_+}^{\vee} \to 0, \\
&0 \to \mcO_{\wX}(1,-1) \to \mcE_{\wX} \to \phi_-^*\mcS_{X_-}^{\vee} \to 0.
\end{align*}
By Lemma~\ref{lem non-split}, these sequences do not split.
Indeed, if the above sequences split, their restriction to the zero-section, which coincide with the sequences in Lemma~\ref{lem non-split}, remain split, but this is a contradiction.

\begin{prop} \label{prop exchange}
There exists an isomorphism ${\phi_{\pm}}_*(\mcE_{\wX}(E)) \simeq \mcP_{\pm}$ for each $+$ and $-$.
\end{prop}

\begin{proof}
Since $\mcO_{\wX}(-1,-1) \simeq \mcO_{\wX}(E)$, twisting the above exact sequences by $\mcO_{\wX}(E)$ gives
\begin{align*}
&0 \to \mcO_{\wX}(-2,0) \to \mcE_{\wX}(E) \to \phi_+^*\mcS_{X_+}^{\vee}(E) \to 0, \\
&0 \to \mcO_{\wX}(0,-2) \to \mcE_{\wX}(E) \to \phi_-^*\mcS_{X_-}^{\vee}(E) \to 0.
\end{align*}
By the projection formula, $R{\phi_+}_*\mcO_{\wX}(-2,0) \simeq \mcO_{X_+}(-2)$ and $R{\phi_-}_*\mcO_{\wX}(0,-2) \simeq \mcO_{X_-}(-2)$.
In addition, since $\phi_{\pm}$ are smooth blowing-ups, $\mcO_E(E)$ is the universal line bundle for 
both $\pi_{\pm} \colon E \to H_{\pm}^0$, and hence $R{\phi_{\pm}}_*\mcO_E(E) = 0$ holds.
Therefore, the standard exact sequence $0 \to \mcO_{\wX} \to \mcO_{\wX}(E) \to \mcO_E(E) \to 0$ shows that $R{\phi_{\pm}}_*\mcO_{\wX}(E) \simeq \mcO_{X_{\pm}}$.
In particular, the projection formula yields an isomorphism
\[ R{\phi_{\pm}}_*\left(\phi_{\pm}^*\mcS_{X_{\pm}}^{\vee}(E)\right) \simeq \mcS_{X_{\pm}}^{\vee} \]
for each $+$ and $-$.
Thus, $R{\phi_{\pm}}_*(\mcE_{\wX}(E)) \simeq {\phi_{\pm}}_*(\mcE_{\wX}(E))$, and this sheaf lies in an exact sequence
\[ 0 \to \mcO_{X_{\pm}}(-2) \to {\phi_{\pm}}_*(\mcE_{\wX}(E)) \to \mcS_{X_{\pm}}^{\vee} \to 0. \]
We claim that this exact sequence does not split for each $+$ and $-$.
To see this, it is enough to show that the two associated homomorphisms
\begin{align*}
\varepsilon_+ \colon &\Ext^1_{\wX}(\phi_+^*\mcS_{X_+}^{\vee}(E), \mcO_{\wX}(-2,0)) \to \Ext_{X_+}^1(\mcS_{X_+}^{\vee}, \mcO_{X_+}(-2)), \\
\varepsilon_- \colon&\Ext^1_{\wX}(\phi_-^*\mcS_{X_-}^{\vee}(E), \mcO_{\wX}(0,-2)) \to \Ext_{X_-}^1(\mcS_{X_-}^{\vee}, \mcO_{X_-}(-2))
\end{align*}
are injective.
Since
\[ \Ext_{X_+}^1(\mcS_{X_+}^{\vee}, \mcO_{X_+}(-2)) \simeq \Ext_{\wX}^1(\phi_+^*\mcS_{X_+}^{\vee}, \mcO_{\wX}(-2,0)),  \]
the homomorphism $\varepsilon_+$ is obtained by applying the functor $\Ext^1_{\wX}(-,\mcO_{\wX}(-2,0))$ to the 
counit morphism
\[ \eta \colon \phi_+^*\mcS_{X_+}^{\vee} \simeq \phi_+^*{\phi_+}_*\phi_+^*\mcS_{X_+}^{\vee}(E) \to \phi_+^*\mcS_{X_+}^{\vee}(E). \]
Note that this $\eta$ is injective and has the cokernel $\phi_+^*\mcS_{X_+}^{\vee}(E)|_E$.
Now
\begin{align*}
\Ext^1_{\wX}(\phi_+^*\mcS_{X_+}^{\vee}(E)|_E, \mcO_{\wX}(-2,0)) &\simeq \Ext^1_E(\phi_+^*\mcS_{X_+}^{\vee}(E)|_E, \mcO_{\wX}(-2,0)|_E \otimes_{\mcO_E} \mcO_E(E)[-1]) \\
&\simeq \Hom_{W}(\pi_+^*\mcS_+^{\vee}, \mcO_{W}(-2,0)) \\
&\simeq \Hom_{H_+}(\mcS_+^{\vee}, \mcO_{H_+}(-2)) = 0.
\end{align*}
Therefore, $\varepsilon_+$ is injective.
The same argument shows that $\varepsilon_-$ is also injective.
Finally, the uniqueness of the non-trivial extensions yields 
\[ {\phi_{\pm}}_*(\mcE_{\wX}(E)) \simeq \mcP_{\pm} \]
as desired.
\end{proof}

\begin{proof}[Proof of Theorem~\ref{tilting exchange}]
Let $\refl(Y)$ denote the category of reflexive coherent sheaves over a normal scheme $Y$.
Since $f_{\pm} \colon X_{\pm} \to \Spec R$ are flopping contractions, they are isomorphic in codimension one,
and hence the push-forwards ${f_{\pm}}_*$ give equivalences of categories
\[ \refl(X_+) \simeq \refl(\Spec R) \simeq \refl(X_-). \]
Since $X_+ \dashrightarrow X_-$ is a flop, ${f_+}_*\mcO_{X_+}(a) \simeq {f_-}_*\mcO_{X_-}(-a) \in \refl(\Spec R)$
for all $a \in \Z$.
In addition, Proposition~\ref{prop exchange} gives
\[ {f_+}_*\mcP_+ \simeq (f_+ \circ \phi_+)_*(\mcE_{\wX}(E)) \simeq (f_- \circ \phi_-)_*(\mcE_{\wX}(E)) \simeq {f_-}_*\mcP_- \in \refl(\Spec R). \]
Since the equivalences commute with duals, ${f_+}_*\mcP_+^{\vee} \simeq {f_-}_*\mcP_-^{\vee}$ also holds.
This completes the proof of the theorem.
\end{proof}

\begin{proof}[Proof of Theorem~\ref{main thm}]
The tilting bundles $\mcT_{\pm}$ give equivalences
\[ \tau_{\pm} := \RHom_{X_{\pm}}(\mcT_{\pm}, -) \colon \Db(\coh X_{\pm}) \xrightarrow{\sim} \Db(\modu \End_{X_{\pm}}(\mcT_{\pm})). \]
The equivalences of categories $\refl(X_+) \simeq \refl(\Spec R) \simeq \refl(X_-)$ together with Theorem~\ref{tilting exchange} imply that
\[ \Lambda := \End_{X_+}(\mcT_+) \simeq \End_R({f_+}_*\mcT_+) \simeq \End_R({f_-}_*\mcT_-) \simeq \End_{X_-}(\mcT_-) \]
as $R$-algebras. 
This yields the equivalences of $R$-linear triangulated categories
\[ \Db(\coh X_+) \xrightarrow[\tau_+]{\sim} \Db(\modu \Lambda) \xrightarrow[\tau_-^{-1}]{\sim} \Db(\coh X_-). \]
The derived equivalence $\tau_+$ sends $\mcO_{X_+}$ to the projective $\Lambda$-module $P_0$ corresponding to the summand $R = {f_+}_*\mcO_{X_+}$ of ${f_+}_*\mcT_+$, and then $\tau_-^{-1}$ sends $P_0$ to $\mcO_{X_-}$.
If 
\[ \Phi := \tau_-^{-1} \circ \tau_+ \colon \Db(\coh X_+) \xrightarrow{\sim} \Db(\coh X_-) \]
denotes the composited derived equivalence,
then for any $\mcF \in \Db(\coh X_+)$,
\begin{align*}
R{f_+}_*\mcF &\simeq \RHom_{X_+}(\mcO_{X_+}, \mcF) 
\simeq \RHom_{X_-}(\Phi \mcO_{X_+}, \Phi \mcF) 
\simeq  \RHom_{X_-}(\mcO_{X_-}, \Phi \mcF) 
\simeq (R{f_-}_* \circ \Phi)(\mcF).
\end{align*}
All the above isomorphisms are functorial, and hence there is a functor isomorphism $R{f_+}_* \simeq  R{f_-}_* \circ \Phi$.
This completes the proof of the main theorem.
\end{proof}

\begin{rem}
Using $\mcT_{\pm}^{\vee}$ instead of $\mcT_{\pm}$ gives another equivalence
\[ \Psi \colon \Db(\coh X_+) \xrightarrow{\sim} \Db(\coh X_-). \]
The composition $\Psi^{-1} \circ \Phi$ gives a non-trivial autoequivalence of $\Db(\coh X_+)$ that fixes five line bundles $\mcO_{X_+}(-2), \mcO_{X_+}(-1), \mcO_{X_+}, \mcO_{X_+}(1), \mcO_{X_+}(2)$.
It would be interesting to study the relation between this autoequivalence and spherical twists.
\end{rem}

During the proof of the main theorem, the following corollary also had been proved.
Note that, for any local model $f \colon \mathcal{X} \to \Spec \mathcal{R}$ of a flopping contraction of type $G_2^{\dagger}$, the completion $\mathfrak{R}$ of $\mathcal{R}$ along the maximal ideal that corresponds to the singularity does not depend on the choice of local models, and $\Spec \mathfrak{R}$ has exactly two crepant resolutions.

\begin{cor}
Two $R$-modules ${f_+}_*\mcT_+ \simeq {f_-}_*\mcT_-$ and ${f_+}_*\mcT_+^{\vee} \simeq {f_-}_*\mcT_-^{\vee}$ give NCCRs of $R$ (cf.~\cite{VdB04, VdB23}).
In particular, the complete local ring $\mathfrak{R}$ of the singularity that appears in a simple flopping contraction of type $G_2^{\dagger}$ admits an NCCR that is derived equivalent to both crepant resolutions.
\end{cor}

Note that two $R$-modules above are both Cohen--Macaulay (cf.~the proof of \cite[Lemma~A.2]{TU10}).

\subsection*{Acknowledgements}
The author would also like to thank Conan Naichung Leung and IMS at CUHK for organizing the workshop that rekindled the author's interest in this research topic.
The author would like to thank Takeru Fukuoka, Akihiro Kanemitsu, Hayato Morimura, Yukinobu Toda, and Michael Wemyss for very helpful conversations,
and anonymous referees for their careful reading and valuable comments.
This work was carried out while the author was a Project Researcher at the Kavli Institute for the Physics and Mathematics of the Universe (Kavli IPMU).
The author would like to express sincere gratitude for the excellent research environment and generous support provided by Kavli IPMU.
This work was supported by World Premier International Research Center Initiative (WPI), MEXT, Japan, 
and by JSPS KAKENHI Grant Number JP24K22829.

\bibliography{flop}
\bibliographystyle{amsplain}

\end{document}